\documentclass[oneside,a4paper,12pt,notitlepage]{article}
\usepackage[left=0.8in, right=0.8in, top=1.5in, bottom=1.5in]{geometry}
\usepackage[T1]{fontenc} 
\usepackage[utf8]{inputenc} 
\usepackage{lipsum} 
\usepackage{lmodern}
\usepackage{amssymb}
\usepackage{amsthm}
\usepackage{bm}
\usepackage{mathtools}
\usepackage{braket}
\usepackage{esint}
\newcommand{\abs}[1]{{\left|#1\right|}}
\newcommand{\norma}[1]{{\left\Vert#1\right\Vert}}

\usepackage{booktabs}
\usepackage{graphicx}
\usepackage{tikz}
\usetikzlibrary{patterns}
\usepackage{multicol}
\usepackage{caption}
\usepackage{enumerate}
\usepackage[skins,theorems]{tcolorbox}
\tcbset{highlight math style={enhanced,
		colframe=black,colback=white,arc=0pt,boxrule=1pt}}
\def\XXint#1#2#3{{\setbox0=\hbox{$#1{#2#3}{\int}$}
    \vcenter{\hbox{$#2#3$}}\kern-.5\wd0}}

\theoremstyle{definition}
\newtheorem{definizione}{Definition}[section]
\theoremstyle{plain}
\newtheorem{teorema}{Theorem}[section]

\newtheorem{lemma}[teorema]{Lemma}

\newtheorem{corollario}[teorema]{Corollary}
\theoremstyle{definition}
\newtheorem{esempio}{Example}[section]
\newtheorem{oss}[esempio]{Remark}

\renewcommand{\div}{\text{div}}

\DeclareMathOperator{\R}{\mathbb{R}}

\makeatletter
\newcommand{\myfootnote}[2]{\begingroup
	\def\@makefnmark{}%
	\addtocounter{footnote}{-1}%
	\footnote{\textbf{#1} #2}%
	\endgroup}
\makeatother
\usepackage{hyperref}
\hypersetup{linktoc=none, bookmarksnumbered, colorlinks=true, linkcolor=red}

\title{Comparison results for solutions to $p$-Laplace equations with Robin boundary conditions}
\author{Vincenzo Amato, Andrea Gentile, Alba Lia Masiello}
\date{\today}
\begin{document}
	\maketitle
	\myfootnote{}{2010 Mathematics Subject Classification: 35J92, 35P15. }
	\myfootnote{}{Key words and phrases: $p$-Lapacian, Robin boundary conditions.}
    \vspace{-0.8cm}
	\begin{abstract}
		In the last decades comparison results of Talenti type for Elliptic Problems with Dirichlet boundary conditions have been widely investigated.
		In this paper, we generalize the results obtained in \cite{ANT} to the case of $p$-Laplace operator  with Robin boundary conditions.

		The point-wise comparison, obtained in \cite{ANT} only in the planar case, holds true in any dimension if $p$ is sufficiently small.
	
	\end{abstract}
	
\section{Introduction}
Let $\beta$ be a positive parameter and let $\Omega$ be a bounded open set of $\R^n$, $n\ge 2$, with Lipschitz boundary. 

\noindent Let $f\in L^{p'}(\Omega)$ be a non-negative function. We consider the following problem
\begin{equation}
\label{p_originale}
\begin{cases}
-\div(\abs{\nabla u}^{p-2} \nabla u)= f & \text{ in } \Omega \\
\abs{\nabla u}^{p-2} \displaystyle{\frac{\partial u}{\partial \nu}} + \beta  \abs{u}^{p-2}u =0  & \text{ on } \partial \Omega.
\end{cases}
\end{equation}
A function $u \in W^{1,p}(\Omega)$ is a weak solution to \eqref{p_originale} if
\begin{equation}
\int_{\Omega} \abs{\nabla u}^{p-2}\nabla u \nabla \varphi \, dx + \beta \int_{\partial \Omega} \abs{u}^{p-2} u \varphi \, d\mathcal{H}^{n-1}(x) = \int_{\Omega} f \varphi \, dx \quad \forall \varphi \in W^{1,p}(\Omega).
\end{equation}
We want to establish a comparison principle with the solution to the following symmetrized problem
\begin{equation}
\label{p_tondo}
\begin{cases}
-\div(\abs{\nabla v}^{p-2} \nabla v)= f^{\sharp} & \text{ in } \Omega^\sharp \\
\abs{\nabla v}^{p-2}  \displaystyle{\frac{\partial v}{\partial \nu}} + \beta \abs{v}^{p-2}  v =0  & \text{ on } \partial \Omega^\sharp,
\end{cases}
\end{equation}
where $\Omega^\sharp$ is the ball centered in the origin with the same measure of $\Omega$ and $f^\sharp$ is the Schwarz rearrangement of $f$ (see next section for its definition).

This kind of problems has been widely investigated in the last decades. The first step is contained in \cite{T}, where Talenti proved a pointwise comparison result between $u^{\sharp}$ and $v$ in the case of Dirichlet Laplacian.
After this, several papers generalized the result of Talenti: for instance, the one by Talenti himself \cite{T2}, in which the operator is a generic non-linear operator in divergence form, or the one by Alvino, Lions and Trombetti \cite{ALT}
in which the authors deal with both elliptic and parabolic cases: in both papers, Dirichlet boundary conditions are considered. 

Different kind of boundary conditions are considered by Alvino, Nitsch and Trombetti in \cite{ANT}, where they establish a comparison between a suitable norm of $u$ and $v$, respectively solution  to 
\begin{equation*}
		\begin{cases}
			-\Delta u= f & \text{in} \, \Omega\\
			   \displaystyle{\frac{\partial u}{\partial \nu}} + \beta  u =0&\text{on} \,  \partial \Omega.
		\end{cases}
		\quad \quad
		\begin{cases}
		    -\Delta v= f^\sharp  & \text{in} \, \Omega^\sharp\\
			   \displaystyle{\frac{\partial u}{\partial \nu}} + \beta  u =0 &\text{on} \,  \partial \Omega^\sharp.
		\end{cases}
	\end{equation*}
They found out that if $f$ is a non-negative function in $L^2(\Omega)$, then
	\begin{align*}
	    \norma{u}_{L^{k,1}(\Omega)}\le \norma{v}_{L^{k,1}(\Omega^\sharp)} \quad & \forall 0<k\le \frac{n}{2n-2}\\
	    \norma{u}_{L^{2k,2}(\Omega)}\le \norma{v}_{L^{2k,2}(\Omega^\sharp)} \quad & \forall 0<k\le \frac{n}{3n-4}
	\end{align*}
where $\norma{\cdot}_{L^{k,q}}$ is the so called \emph{Lorentz norm}, whose definition can be found in next section. Moreover, the authors in \cite{ANT} were able to establish a comparison \'a la Talenti, 
$$ u^\sharp(x)\le v(x), \quad \forall x \in \Omega^\sharp  $$
in the case $f\equiv 1$ and $n=2$.
This will be the starting point of our work: it will be clear that our results coincide with the one in \cite{ANT} in the case $p=2$.

For completeness sake, we cite that this wasn't the first result in this sense, indeed in \cite{BG2} the authors study a comparison result for the $p$-torsion, that is the case $f \equiv 1$, with a completely different argument, obtaining
$$\norma{u}_{L^1{\Omega}} \le \norma{v}_{L^1(\Omega^\sharp)}.$$

Another work that is worth to be mentioned is \cite{ACNT}, where the authors obtained similar results to \cite{ANT} in the case of mixed Dirichlet and Robin boundary conditions. 
 
This paper is organized as follows. In the next section, we give some basic notions about rearrangements of functions and Lorentz spaces. Moreover, we list some properties of the solutions to problems \eqref{p_originale} and \eqref{p_tondo}. In section \ref{section_2}, we prove the main results about comparison of the two solutions in terms of the Lorentz norm. 

In particular, we prove
\begin{teorema}
    \label{teorema1}
	Let $u$ and $v$ be the solutions to problem \eqref{p_originale} and \eqref{p_tondo} respectively. Then we have
	\begin{equation}
	    \label{diseq_f_generica}
	    \norma{u}_{L^{k,1}(\Omega)} \, \leq \norma{v}_{L^{k,1}(\Omega^\sharp)} \, \; \forall \, 0 < k \leq \frac{n(p-1)}{(n-1)p},
	\end{equation}
	\begin{equation}
	    \label{fgen2}
	    \norma{u}_{L^{pk,p}(\Omega)} \, \leq \norma{v}_{L^{pk,p}(\Omega^\sharp)} \, \; \forall \, 0 < k \leq \frac{n(p-1)}{(n-2)p +n}.
	\end{equation}
\end{teorema}
We observe that from  Theorem \ref{teorema1}, we have that, if $p \geq n$
\begin{equation*}
    \norma{u}_{L^1(\Omega)}\le \norma{v}_{L^1(\Omega^\sharp)} \quad \text{and} \quad \norma{u}_{L^p(\Omega)}\le \norma{v}_{L^p(\Omega^\sharp)}.
\end{equation*}

\begin{teorema}
\label{teorema2}
	 Assume that $f\equiv 1$ and let $u$ and $v$ be the solutions to \eqref{p_originale} and \eqref{p_tondo} respectively. 
	\begin{enumerate}[$(i)$]
		\item If $\displaystyle{1 \leq p \leq \frac{n}{n-1}}$ then
		\begin{equation}
		\label{7}
		u^\sharp (x) \leq v(x) \qquad x \in \Omega^\sharp,
		\end{equation}
		\item if $\displaystyle{ p > \frac{n}{n-1}}$ \, and \, $ \displaystyle{0 <k \leq \frac{n(p-1)}{n(p-1)-p}}$ , then
		\begin{equation}
		\begin{aligned}
		\norma{u}_{L^{k,1}(\Omega)} &\leq \norma{v}_{L^{k,1}(\Omega^{\sharp})} \\ 
		\norma{u}_{L^{pk,p}(\Omega)} &\leq \norma{v}_{L^{pk,p}(\Omega^{\sharp})} .
		\end{aligned}
		\end{equation} 
	\end{enumerate}
\end{teorema}

Then we explicitly observe that in the case $f\equiv 1$ from Theorem \ref{teorema2}   we have that 
\begin{equation*}
    \norma{u}_{L^1(\Omega)}\le \norma{v}_{L^1(\Omega^\sharp)} \quad \text{and} \quad \norma{u}_{L^p(\Omega)}\le \norma{v}_{L^p(\Omega^\sharp)} \qquad \text{ for } p>1.
\end{equation*}

While we have the point-wise comparison only for $\displaystyle{p \leq \frac{n}{n-1}}$.

In section \ref{section_3}, using tools from section \ref{section_2}, we give a new proof of the Faber-Krahn inequality with Robin boundary conditions in the case $p \geq n$.  This topic was already studied in the papers by Bucur, Giacomini, Daners and Trebeschi, \cite{BD}, \cite{BG} , \cite{BG2} and \cite{BGT} where the authors proved the Faber-Krahn inequality for the eigenvalues of the Laplacian, or  of the $p$-Laplacian, with Robin boundary conditions, for every $p >1$. Actually, the results in \cite{BD} are more general, since they hold for every $p>1$, but they are obtained with completely different tools than the ones contained in our paper.

Finally, in section \ref{section_4}, we provide some examples and open problems and we discuss the optimality of our results. 

\section{Notions and preliminaries}
\begin{definizione}
	Let $u: \Omega \to \R$ be a measurable function, the \emph{distribution function} of $u$ is the function $\mu : [0,+\infty[\, \to [0, +\infty[$ defined by
	$$
	\mu(t)= \abs{\Set{x \in \Omega \, :\,  \abs{u(x)} > t}}
	$$
\end{definizione}
Here, and in the whole paper, $\abs{A}$ stands for the $n$-dimensional Lebesgue measure of the set $A$.
\begin{definizione}
	Let $u: \Omega \to \R$ be a measurable function, the \emph{decreasing rearrangement} of $u$, denoted by $u^\ast$, is the distribution function of $\mu $. 
	
	\noindent The \emph{Schwarz rearrangement} of $u$ is the function $u^\sharp $ whose level sets are balls with the same measure as the level sets of $u$. The functions $u^\sharp$ and $u^*$ are linked by the relation
	$$u^\sharp (x)= u^*(\omega_n \abs{x}^n)$$
\end{definizione}
\noindent It is easily checked that $u$, $u^*$ e $u^\sharp$ are equi-distributed, so it follows that
$$ \displaystyle{\norma{u}_{L^p(\Omega)}=\norma{u^*}_{L^p(0, \abs{\Omega})}=\lVert{u^\sharp}\rVert_{L^p(\Omega^\sharp)}}.$$
An important propriety of the decreasing rearrangement is the Hardy- Littlewood inequaliy, that is
\begin{equation*}
 \int_{\Omega} \abs{h(x)g(x)} \, dx \le \int_{0}^{\abs{\Omega}} h^*(s) g^*(s) \, ds.
\end{equation*}

\noindent So, by choosing $h=\chi_{\left\lbrace\abs{u}>t\right\rbrace}$, one has

\begin{equation*}
\int_{\abs{u}>t} \abs{h(x)} \, dx \le \int_{0}^{\mu(t)} h^*(s) \, ds.
\end{equation*}

\begin{definizione}
Let $0<p<+\infty$ and $0<q\le +\infty$. The Lorentz space $L^{p,q}(\Omega)$ is the space of those functions such that the quantity:
\begin{equation*}
    \norma{g}_{L^{p,q}} =
    \begin{cases}
   	\displaystyle{ p^{\frac{1}{q}} \left( \int_{0}^{\infty}  t^{q} \mu(t)^{\frac{q}{p}}\, \frac{dt}{t}\right)^{\frac{1}{q}}} & 0<q<\infty\\
	 \displaystyle{\sup_{t>0} \, (t^p \mu(t))} & q=\infty
	\end{cases}
\end{equation*}
is finite.
\end{definizione}
	
\noindent Let us observe that for $p=q$ the Lorentz space coincides with the  $L^p$ space, as a consequence of the well known \emph{Cavalieri's Principle}

$$\int_\Omega \abs{g}^p =p \int_0^{+\infty} t^{p-1} \mu(t) \, dt.$$
	
\noindent See \cite{T3} for more details on Lorentz space.

 Let us consider the functional
\begin{equation*}
	\mathfrak{F}(w)=\frac{1}{p} \int_{\Omega} \abs{\nabla w}^p \, dx +\frac{\beta}{p}\int_{\partial\Omega} \abs{w}^p \, d\mathcal{H}^{n-1}(x) - \int_{\Omega} fw \, dx	
\end{equation*}
defined on $W^{1,p}(\Omega)$. This functional is well defined and its Euler-Lagrange equation is exactly \eqref{p_originale}. If we show that the functional admits a minimum, our problem will always have a solution.
\begin{enumerate}[$1)$]
    \item Let us show that the functional is bounded from below, indeed using the parametric Young inequality, we have
	\begin{equation*}
		\begin{aligned}
		\mathfrak{F}(u)&\ge \frac{1}{p} \int_{\Omega} \abs{\nabla u}^p \, dx + \frac{\beta}{p}\int_{\partial\Omega} \abs{u}^p \, d\mathcal{H}^{n-1}(x) - \frac{\varepsilon^p}{p} \int_{\Omega} \abs{u}^p \, dx -  \frac{1}{p' \varepsilon^{p'}}\int_{\Omega} \abs{f}^{p'} \, dx\\
		&\ge  \frac{1}{p} \left( \int_{\Omega} \abs{\nabla u}^p \, dx+ \beta \int_{\partial\Omega} \abs{u}^p\, d\mathcal{H}^{n-1}(x)\right) - \frac{\varepsilon^p}{p} \int_{\Omega} \abs{u}^p \, dx 
		-\frac{1}{p' \varepsilon^{p'}}\int_{\Omega} \abs{f}^{p'} \, dx\\
		&\ge  \frac{ \lambda_{1,\beta}(\Omega)- \varepsilon^p}{p} \int_{\Omega} \abs{u}^p \, dx -  \frac{1}{p' \varepsilon^{p'}}\int_{\Omega} \abs{f}^{p'} \, dx
		\end{aligned}
	\end{equation*}
	In the last inequality we used the Sobolev inequality with trace term $$
	\int_\Omega \abs{\nabla u}^p +\beta \int_{\partial \Omega } \abs{u}^p   \geq \lambda_{1,\beta}(\Omega) \int_\Omega \abs{u}^p.
	$$
	In general, the quantity $\lambda_{1,\beta}(\Omega)$ denotes the first eigenvalue of the $p$-Laplacian with Robin boundary conditions, whose definition is given in \eqref{trace}, which can be also seen as a trace constant of the set $\Omega$.
	
	If $\varepsilon$ is small enough, then the quantity
	$$ \frac{ \lambda_{1,\beta}(\Omega)- \varepsilon^p}{p}$$
	is non negative, and then
	\begin{equation*}
		\mathfrak{F}(u)\ge - \frac{1}{p' \varepsilon^{p'}}\int_{\Omega} \abs{f}^{p'}
	\end{equation*}
	so
	\begin{equation*}
		m=\inf_{W^{1,p}} \mathfrak{F}(u)> - \infty.
	\end{equation*}
	
	\item Compactness and lower semicontinuity.
	
	Let $\left\lbrace u_i\right \rbrace$ be a minimizing sequence. We can assume that $\mathfrak{F}(u_i)\le m+1$, $\forall i$. Using again the Young inequality, we have
	\begin{equation*}
		\begin{aligned}
		m+1 &\geq \frac{1}{p} \int_{\Omega} \abs{\nabla u_i}^p \, dx + \frac{\beta}{p}\int_{\partial\Omega} \abs{u_i}^p \, d\mathcal{H}^{n-1}(x)- \int_{\Omega} fu_i \, dx \\
		&\geq \frac{1}{p} \int_{\Omega} \abs{\nabla u_i}^p \, dx + \frac{\beta}{p}\int_{\partial\Omega} \abs{u_i}^p \, d\mathcal{H}^{n-1}(x)- \frac{\varepsilon^p}{p}\int_{\Omega} \abs{u_i}^p\, dx- \frac{1}{p' \varepsilon^{p'}}\int_{\Omega} \abs{f}^{p'} \, dx
		\end{aligned}
	\end{equation*}
	Then
	\begin{equation*}
	\begin{aligned}
	m+1 + \frac{1}{p' \varepsilon^{p'}}\int_{\Omega} \abs{f}^{p'} \, dx &\geq  \frac{1}{2p} \left( \int_{\Omega} \abs{\nabla u_i}^p \, dx +\beta \int_{\partial\Omega} \abs{u_i}^p \, d\mathcal{H}^{n-1}(x)\right)-\frac{\varepsilon^p}{p}\int_{\Omega} \abs{u_i}^p\, dx \\
	  &+\frac{1}{2p} \left( \int_{\Omega} \abs{\nabla u_i}^p \, dx +\beta \int_{\partial\Omega} \abs{u_i}^p \, d\mathcal{H}^{n-1}(x)\right) \\
	&\geq \frac{1}{2p} \int_{\Omega} \abs{\nabla u_i}^p \, dx + \left(\frac{\lambda_{1,\beta}(\Omega)- 2 \varepsilon^p}{2p}\right)\int_{\Omega} \abs{u_i}^p\, dx.
	\end{aligned}		
	\end{equation*}
	
	Then, the minimizing sequence $\left\lbrace u_i \right\rbrace$ is bounded in $W^{1,p}(\Omega)$, so there exists a subsequence $\Set{u_{i_k}}$ weakly converging in $W^{1,p}(\Omega)$ and strongly in $L^p(\Omega)$ to a function $u$. Let us show that $u$ is the minimum.
	
	The function $t^p$ is strictly convex for $p>1$, so
	\begin{gather}
	    \label{2} \abs{ u_{i_k}}^p \geq \abs{ u}^p + p \abs{u}^{p-2}  u( u_{i_k}-u) \\
	    \label{1} \abs{\nabla u_{i_k}}^p \geq \abs{\nabla u}^p + p \abs{\nabla u}^{p-2} \nabla u(\nabla u_{i_k}-\nabla u)
	\end{gather}

	Putting \eqref{2} e \eqref{1} in $\mathfrak{F}(u_{i_k})$, we obtain
	\begin{equation*}
	\begin{aligned}
			\int_{\Omega} f u_{i_k} \, dx + \mathfrak{F}(u_{i_k})& \geq \frac{1}{p}\int_{\Omega}\abs{\nabla u}^p \, dx + \int_{\Omega} \abs{\nabla u}^{p-2} \nabla u(\nabla u_{i_k}-\nabla u)\\
			&+\frac{\beta}{p} \int_{\partial\Omega} \abs{ u}^p \, d\mathcal{H}^{n-1}(x) + \beta \int_{\partial\Omega} \abs{u}^{p-2}  u( u_{i_k}-u) \, d\mathcal{H}^{n-1}
		\end{aligned}
	\end{equation*}
	\noindent Passing to the limit for $k\to \infty$, by the weak convergence of $\left\lbrace u_{i_k}\right\rbrace$ the integral over $\Omega$ on the right-hand side goes to 0. The integral over $\partial\Omega$ goes to 0 as well. Indeed, the space $W^{1,p}(\Omega)$ is compactly embedded in $L^p(\partial\Omega)$ (for more details, see \cite{N} 2.5), and $u_{i_k}-u\to 0$ in $L^p(\partial\Omega)$. So, we obtain
	\begin{equation*}
		m \geq \mathfrak{F}(u).
	\end{equation*} 
	\noindent This ensures us that $u$ is the minimum of the functional. 
	
	\noindent The uniqueness of the minimum follows from the fact that $ \mathfrak{F}(u)$ is the sum of a strictly convex part and a linear part.
\end{enumerate}

We observe that the solutions $u$ and $v$  to \eqref{p_originale} and \eqref{p_tondo} respectively are both $p$-superharmonic and then, by the strong maximum principle in \cite{V}, it follows that they achieve their minima on the boundary.
Denoting by $u_m$ and $v_m$ the minimum of $u$ and $v$ respectively, thanks to the positiveness of $\beta$ and the Robin boundary conditions, we have that $u_m \geq 0$ and $v_m \geq 0$. Hence $u$ and $v$ are strictly positive in the interior of $\Omega$.
Moreover we can observe that 
\begin{equation}
		\label{minima_eq}
		u_m = \min_\Omega u \leq  \min_{\Omega^\sharp} v= v_m,
	\end{equation}
in fact, if we consider

	\begin{equation*}
		\begin{split}
			 v_m^{p-1}  \text{Per}(\Omega^\sharp) &= \int_{\partial \Omega^\sharp} v(x)^{p-1} \, d\mathcal{H}^{n-1}(x)= \frac{1}{\beta}\int_{\Omega^\sharp} f^\sharp \, dx=\frac{1}{\beta} \int_{\Omega} f \, dx \\
			& = \int_{\partial \Omega} u(x)^{p-1} \, d\mathcal{H}^{n-1}(x) \\
			&\geq u_m^{p-1}  \text{Per}(\Omega)  \geq \abs u_m^{p-1} \text{Per}(\Omega^\sharp).
		\end{split}
	\end{equation*}
	A consequence of \eqref{minima_eq} that will be used in what follows is that
	\begin{equation}
		\label{mf}
		\mu (t) \leq \phi (t) = \abs{\Omega} \quad \forall t \leq v_m.
	\end{equation} 

\subsection{Useful lemmas}	
Let $u$ be the solution to \eqref{p_originale}. For $t\geq 0$, we denote by
$$U_t=\left\lbrace x\in \Omega : u(x)>t\right\rbrace \quad \partial U_t^{int}=\partial U_t \cap \Omega, \quad \partial U_t^{ext}=\partial U_t \cap \partial\Omega,$$
and by
$$\mu(t)=\abs{U_t} \quad P_u(t)=Per(U_t)$$
where $\abs{\cdot}$ is the Lebesgue measure on $\R^n$ and $Per(\cdot)$ is the perimeter. \\
If $v$ is the solution to \eqref{p_tondo}, using the same notations, we set
$$V_t=\left\lbrace x\in \Omega^\sharp : v(x)> t\right\rbrace, \quad \phi(t)=\abs{V_t}, \quad P_v(t)=Per(V_t).$$

Because of the invariance of the $p$-Laplacian and of the Schwarz rearrangement of $f$ by rotation, there exists a radial solution to \eqref{p_tondo} and, by uniqueness of solutions, this solution is $v$.

Since $v$ is radial, positive and decreasing along the radius then, for $0\le t\le v_m$, $V_t=\Omega^\sharp$, while, for $v_m<t<\max_{\Omega^\sharp}v$, $V_t$ is a ball, concentric to $\Omega^\sharp$ and strictly contained in it.

\begin{lemma}[Gronwall]
	\label{lemma_Gronwall}
	Let $\xi(t): [\tau_0 , + \infty[ \,\to \R$ be a continuous and differentiable function satisfying, for some non negative constant $C$, the following differential inequality
	$$
	\tau \xi' (\tau ) \leq (p-1) \xi(\tau) + C \quad \forall \tau \geq \tau_0 >0.
	$$
	
	\noindent Then we have
	\begin{itemize}
		\item[(i)] $\displaystyle{ \xi(\tau) \leq \left(\xi(\tau_0) + \frac{C}{p-1}\right) \left( \frac{\tau}{\tau_0}\right)^{p-1} - \frac{C}{p-1}   \quad \forall \tau \geq \tau_0}$;
		\item[(ii)]$\displaystyle{ \xi'(\tau) \leq \left( \frac{(p-1)\xi(\tau_0 )+ C}{\tau_0}\right) \left( \frac{\tau}{\tau_0}\right)^{p-2} \quad \forall \tau \geq \tau_0}$.
	\end{itemize}
\end{lemma}
\begin{proof}
    Dividing both sides of the differential inequality by $\tau^p$, we obtain
	\begin{equation*}
	\left(\frac{\xi'(\tau)}{\tau^{p-1}}-(p-1) \frac{\xi(\tau)}{\tau^p}\right) = \left(\frac{\xi(\tau)}{\tau^{p-1}}\right)' \leq \frac{C}{\tau^p}.
	\end{equation*}
	Now, we integrate from $\tau_0$ to $\tau$ and we obtain
	\begin{equation*}
	\begin{split}
	\int_{\tau_0}^\tau \left(\frac{\xi(\tau)}{\tau^{p-1}}\right)' \, d\tau &\leq \int_{\tau_0}^\tau \frac{C}{\tau^p} \, d\tau  \\
	\implies \xi(\tau) &\leq \left(\xi(\tau_0) + \frac{C}{p-1}\right) \left( \frac{\tau}{\tau_0}\right)^{p-1} - \frac{C}{p-1},
	\end{split}
	\end{equation*}
	which gives \emph{(i)}.
	
	\noindent In order to obtain \emph{(ii)}, we just take into account \emph{(i)} in the differential inequality.
\end{proof}

\begin{lemma} Let $u$ and $v$ be solutions to \eqref{p_originale} and \eqref{p_tondo} respectively. Then for almost every $t >0$ we have
	\begin{equation} 
	\label{3.2}
	\gamma_n \mu(t)^{\left(1-\frac{1}{n}\right)\frac{p}{p-1}} \leq \left(\int_0^{\mu (t)}f^\ast (s ) \, ds\right)^{\frac{1}{p-1}} \left( - \mu'(t) + \frac{1}{\beta ^{\frac{1}{p-1}}}\int_{\partial U_t^\text{ext}} \frac{1}{u} \, d\mathcal{H}^{n-1}(x)\right),
	\end{equation}
	where  $\displaystyle{\gamma_n= \left(n \omega_n^{1/n}\right)^{\frac{p}{p-1}}}$.
	
	\noindent And 
	\begin{equation} 
	\label{3.1}
	\gamma_n \phi(t)^{\left(1-\frac{1}{n}\right)\frac{p}{p-1}} = \left(\int_0^{\phi (t)}f^\ast (s ) \, ds\right)^{\frac{1}{p-1}} \left( - \phi'(t) + \frac{1}{\beta ^{\frac{1}{p-1}}}\int_{\partial V_t^\text{ext}} \frac{1}{v} \, d\mathcal{H}^{n-1}(x)\right).
	\end{equation}
\end{lemma}

\begin{proof} Let $t >0$ e $h >0$, we choose the test function 
	\begin{equation*}
	\left.
	\varphi (x)= 
	\right.
	\begin{cases}
	0 & \text{ if }u < t \\
	u-t & \text{ if }t< u < t+h \\
	h & \text{ if }u > t+h.
	\end{cases}
	\end{equation*}
	
	\noindent Then,
	\begin{equation*}
	\begin{split}
	\int_{U_t \setminus U_{t+h}} \abs{\nabla u}^p\, dx &+ \beta h \int_{\partial U_{t+h}^{ ext }} u^{p-1}  \, d\mathcal{H}^{n-1}(x) + \beta \int_{\partial U_{t}^{ ext }\setminus \partial U_{t+h}^{ ext }} u^{p-1} (u-t) \, d\mathcal{H}^{n-1}(x) \\
	& =  \int_{U_t \setminus U_{t+h}} f (u-t ) \, dx + h \int_{U_{t+h}} f \, dx.
	\end{split}
	\end{equation*}
	Dividing by $h$, using coarea formula and letting $h$ go to 0, we have that for a. e. $t>0$
	\begin{equation*}
	\int_{\partial U_t} g(x) \, d\mathcal{H}^{n-1}(x) = \int_{U_{t}} f\, dx ,
	\end{equation*}
	where
	$$
	\left.
	g(x)= 
	\right.
	\begin{cases}
	\abs{\nabla u }^{p-1} & \text{ if }x \in \partial U_t^{ int },\\
	\beta u ^{p-1}& \text{ if }x \in \partial U_t^{ ext }
	.
	\end{cases}
	$$
	
	\noindent So, using the isoperimetric inequality, for a. e. $t>0$ we have 
	\begin{equation*}
	\begin{split}
	n \omega_n^{1/n} \mu(t)^{\left(1 - \frac{1}{n}\right)} &\leq P_u(t) = \int_{\partial U_t} \,  d\mathcal{H}^{n-1}(x)\leq \left(\int_{\partial U_t}g\, d\mathcal{H}^{n-1}(x)\right)^{\frac{1}{p}} \left(\int_{\partial U_t}\frac{1}{g^{\frac{1}{p-1}}} \, d\mathcal{H}^{n-1}(x)\right)^{1-\frac{1}{p}} \\
	&= \left(\int_{\partial U_t}g \, d\mathcal{H}^{n-1}(x)\right)^{\frac{1}{p}} \left( \int_{\partial U_t^{ int }}\frac{1}{\abs{\nabla u}}\, d\mathcal{H}^{n-1}(x) +\frac{1}{\beta^{\frac{1}{p-1}}} \int_{\partial U_t^{ ext }}\frac{1}{u} \,  d\mathcal{H}^{n-1}(x) \right)^{1-\frac{1}{p}} \\
	&\leq \left(\int_0^{\mu(t)} f^\ast (s) \, ds\right)^{\frac{1}{p}} \left( -\mu'(t) +\frac{1}{\beta^{\frac{1}{p-1}}} \int_{\partial U_t^{ ext }}\frac{1}{u} \,  d\mathcal{H}^{n-1}(x) \right)^{1-\frac{1}{p}} \quad t \in [0, \max_{\Omega} u ),
	\end{split}
	\end{equation*}
	Then \eqref{3.2} follows. We notice that if $v$ is the solution to \eqref{p_tondo}, than all the inequalities are verified as equalities, so we have \eqref{3.1}.
\end{proof}
\begin{lemma}
	\label{lemma3.3}
	For all $\tau \geq v_m $ we have
	\begin{equation}
	\label{3.11}
	\int_0^\tau t^{p-1} \left(\int_{\partial U_t^{ ext } } \frac{1}{ u(x) } \, d \mathcal{H}^{n-1}(x)\right) \, dt \leq  \frac{1}{p\beta} \int_0^{\abs{\Omega}} f^\ast(s) \,ds.
	\end{equation}
	Moreover,
	\begin{equation}
	\label{3.10}
	\int_0^\tau t^{p-1} \left(\int_{\partial V_t \cap \partial\Omega^\sharp } \frac{1}{ v(x) } \, d \mathcal{H}^{n-1}(x)\right) \, dt =  \frac{1}{p\beta} \int_0^{\abs{\Omega}} f^\ast(s) \,ds,
	\end{equation}
\end{lemma}
\begin{proof}

	If we integrate the quantity 
	$$t^{p-1} \left(\int_{\partial U_t^{ ext } } \frac{1}{ u(x) } \, d \mathcal{H}^{n-1}(x)\right), $$
	from 0 to $+\infty$, by Fubini theorem, we obtain
	\begin{equation*}
	\begin{split}
	\int_0^\infty \tau^{p-1} \left(\int_{\partial U_\tau^{ ext } } \frac{1}{ u(x) } \, d \mathcal{H}^{n-1}(x)\right) \, d\tau&  =\int_{\partial \Omega} \left(\int_0^{u(x)} \frac{\tau^{p-1}}{u(x)} \, d\tau \right)\, d\mathcal{H}^{n-1}(x)\\
	&=\frac{1}{p}\int_{\partial \Omega} u(x)^{p-1}  \, d\mathcal{H}^{n-1}(x) \\
	&=  \frac{1}{p \beta} \int_0^{\abs{\Omega}}f^\ast (s) \, ds,
	\end{split}
	\end{equation*}
	where the last equality follows from the fact that $u$ solves \eqref{p_originale}.
	
	\noindent Analogously  
	\begin{equation*}
	\int_0^\infty \tau^{p-1} \left(\int_{\partial V_\tau \cap \partial\Omega^{\sharp} } \frac{1}{ v(x) } \, d \mathcal{H}^{n-1}(x)\right) \, d\tau = \frac{1}{p \beta} \int_0^{\abs{\Omega}}f^\ast (s) \, ds.
	\end{equation*}
	
	\noindent Since $u$ is positive, we obtain, $\forall t\ge0$,
	$$
	\int_0^t \tau^{p-1} \left(\int_{\partial U_\tau^{ ext } } \frac{1}{ u(x) } \, d \mathcal{H}^{n-1}(x)\right) \, d\tau \leq\frac{1}{p \beta} \int_0^{\abs{\Omega}}f^\ast (s) \, ds,
	$$
	on the other hand, since $\partial V_t \cap \partial \Omega^\sharp$ is empty for $t\geq v_m$, we have
	$$
	\int_0^t \tau^{p-1} \left(\int_{\partial V_\tau \cap \partial\Omega^{\sharp} } \frac{1}{ v(x) } \, d \mathcal{H}^{n-1}(x)\right) \, d\tau = \frac{1}{p \beta} \int_0^{\abs{\Omega}}f^\ast (s) \, ds.
	$$  
	and the proof of lemma \ref{lemma3.3} is complete.
\end{proof}
\begin{oss}
	\label{oss}
	It can be observed that, since $\partial V_t \cap \partial \Omega^\sharp$ is empty for $t\geq v_m$ and $\phi (t ) = \abs{\Omega}$ for $t \leq v_m $, for all $\delta>0$ and for all $t$, we have
	\begin{equation*}
	\begin{split}
	&\int_0^t \tau^{p-1}\phi(\tau)^\delta \left(\int_{\partial 	V_\tau \cap \partial\Omega^{\sharp} } \frac{1}{ v(x) } \, d \mathcal{H}^{n-1}(x)\right) \, d\tau = \\ 
	&\int_0^{v_m} \tau^{p-1}\phi(\tau)^\delta \left(\int_{\partial 	V_\tau \cap \partial\Omega^{\sharp} } \frac{1}{ v(x) } \, d \mathcal{H}^{n-1}(x)\right) \, d\tau = \\
	&\int_0^{+\infty} \tau^{p-1}\phi(\tau)^\delta \left(\int_{\partial 	V_\tau \cap \partial\Omega^{\sharp} } \frac{1}{ v(x) } \, d \mathcal{H}^{n-1}(x)\right) \, d\tau =
	\frac{\abs{\Omega}^\delta}{p \beta} \int_0^{\abs{\Omega}}f^\ast (s) \, ds.
	\end{split}
	\end{equation*}
\end{oss}

\section{Main results}
\label{section_2}

Now we prove Theorem \ref{teorema1} and Theorem \ref{teorema2} .
\begin{proof}[Proof of Theorem \ref{teorema1}]
	Let $\displaystyle{0< k \leq \frac{n(p-1)}{p(n-1)}}$, so $\displaystyle{\delta = \frac{1}{k} -\frac{(n-1)p}{n(p-1)}} $ is positive.
	
	\noindent Multiplying \eqref{3.2} by $t^{p-1} \mu (t)^\delta$ and integrating from $0 $ to $\tau\geq v_m$, by the previous Lemma, we obtain
	\begin{equation}
		\label{3.12}
		\begin{split}
			\int_0^\tau \gamma_n t^{p-1}\mu(t)^{\frac{1}{k}} \, dt &\leq  	\int_0^\tau \left(- \mu'(t) \right) t^{p-1}\mu(t)^\delta \left(\int_0^{\mu (t)}f^\ast (s ) \, ds\right)^{\frac{1}{p-1}} \, dt \\ 
			&+ \frac{\abs{\Omega}^\delta}{p\beta ^{\frac{p}{p-1}}} \left(\int_0^{\abs{\Omega}}f^\ast (s ) \, ds\right)^{\frac{p}{p-1}}.	
		\end{split} 
	\end{equation}

\noindent Setting $\displaystyle{F(l)= \int_0^l \omega^\delta \left(\int_0^\omega f^\ast (s) \, ds \right)^{\frac{1}{p-1}} \, d\omega}$, we can integrate by parts both sides of the last inequality, getting
\begin{equation*}
	\begin{split}
		\tau^{p-1} \left(\left(\int_0^\tau \gamma_n \mu(t)^{\frac{1}{k}} \, dt\right)+ F(\mu(\tau)) \right) &\leq  (p-1)  \int_0^\tau t^{p-2} \left(\left(\int_0^t \gamma_n \mu(s)^{\frac{1}{k}}\, ds\right)+ F(\mu(t))\right)\, dt \\ 
		&+\frac{\abs{\Omega}^\delta}{p\beta ^{\frac{p}{p-1}}} \left(\int_0^{\abs{\Omega}}f^\ast (s ) \, ds\right)^{\frac{p}{p-1}}.
	\end{split} 
\end{equation*}

\noindent Setting $\displaystyle{\xi(\tau)= \int_0^\tau t^{p-2} \left(\int_0^t \gamma_n \mu(s)^{\frac{1}{k}}\, ds+ F(\mu(t))\right)\, dt}$ and $\displaystyle{C=\frac{\abs{\Omega}^\delta}{p\beta ^{\frac{p}{p-1}}} \left(\int_0^{\abs{\Omega}}f^\ast (s ) \, ds\right)^{\frac{p}{p-1}}}$, we are in the hypothesis of Lemma \ref{lemma_Gronwall} (Gronwall), namely
$$
\tau \xi'(\tau ) \leq (p-1)\xi(\tau) +C,
$$
so, choosing $\tau_0= v_m$, we have
\begin{equation*}
	\begin{multlined}
		\tau^{p-2} \left(\int_0^\tau \gamma_n \mu(s)^{\frac{1}{k}}\, ds+ F(\mu(\tau))\right) \leq \left( \frac{(p-1)\xi(v_m )+ C}{v_m}\right) \left( \frac{\tau}{v_m}\right)^{p-2},
	\end{multlined} 
\end{equation*}
where
$$
\xi(v_m)= \int_0^{v_m} t^{p-2} \left(\int_0^t \gamma_n \mu(s)^{\frac{1}{k}}\, ds+ F(\mu(t))\right)\, dt.
$$
\noindent The previous inequality becomes an equality if we replace $\mu(t)$ with $\phi(t)$. Since $\mu(t) \leq \phi(t)=\abs{\Omega}, \quad \forall t \leq v_m$,  and $F(l)$ is monotone, we 
obtain
$$
\int_0^{v_m} t^{p-2} \left(\int_0^t \gamma_n \mu(s)^{\frac{1}{k}}\, ds+ F(\mu(t))\right)\, dt\leq \int_0^{v_m} t^{p-2} \left(\int_0^t \gamma_n \phi(s)^{\frac{1}{k}}\, ds+ F(\phi(t))\right)\, dt,
$$
hence
$$
\int_0^\tau \gamma_n \mu(s)^{\frac{1}{k}}\, ds+ F(\mu(\tau))\leq \int_0^\tau \gamma_n \phi(s)^{\frac{1}{k}}\, ds+ F(\phi(\tau)).
$$

\noindent Passing to the limit as $\tau\to \infty$, we get
$$
\int_0^\infty \mu(t)^{\frac{1}{k}} \, dt \leq \int_0^\infty \phi(t)^{\frac{1}{k}} \, dt ,
$$
and hence
$$
\norma{u}_{L^{k,1}(\Omega)} \leq \norma{v}_{L^{k,1}(\Omega^\sharp)} \quad \forall \, 0 \, < k \leq \frac{n(p-1)}{p(n-1)} .
$$

\noindent To prove the inequality \eqref{fgen2}, it is enough to show that
\begin{equation}
	\label{basta}
		\int_0^\infty t^{p-1} \mu(t)^{\frac{1}{k}} \, dt \leq \int_0^\infty  t^{p-1} \phi(t)^{\frac{1}{k}} \, dt.
\end{equation}

\noindent Let us consider equation \eqref{3.12}, let us integrate by parts the first term on the right-hand side from 0 to $\tau$ and then let us pass to the limit as $\tau\to \infty$, we have
$$ 
\int_0^\infty \gamma_n t^{p-1}\mu(t)^{\frac{1}{k}} \, dt \leq  	(p-1)\int_0^\infty t^{p-2} F(\mu(t)) \, dt + \frac{\abs{\Omega}^\delta}{p\beta ^{1+ \frac{1}{p-1}}} \left(\int_0^{\abs{\Omega}}f^\ast (s ) \, ds\right)^{\frac{p}{p-1}}.
$$
Therefore, if we show that
\begin{equation}
    \label{scorciatoia}
    \int_0^\infty t^{p-2} F(\mu(t)) \, dt \leq \int_0^\infty t^{p-2} F(\phi(t)) \, dt
\end{equation}
we obtain the \eqref{basta}. To this aim, we multiply \eqref{3.2} by $\displaystyle{t^{p-1}F(\mu(t)) \mu(t)^{-\frac{(n-1)p}{n(p-1)} }}$ and integrate. First, we observe that, by the choice $\displaystyle{k \leq \frac{n(p-1)}{(n-2)p+n}}$, it follows that the function $\displaystyle{h(l)= F(l)l^{-\frac{(n-1)p}{n(p-1)}}}$  is non decreasing. Hence, we obtain
\begin{equation*}
	\begin{split}
		\int_0^\tau \gamma_n t^{p-1}F(\mu(t)) \, dt &\leq  	\int_0^\tau \Bigl(- \mu'(t) \Bigr) t^{p-1}\mu(t)^{-\frac{(n-1)p}{n(p-1)} } F(\mu(t)) \left(\int_0^{\mu (t)}f^\ast (s ) \, ds\right)^{\frac{1}{p-1}} \, dt \\
		&+F(\abs{\Omega})\frac{\abs{\Omega}^{-\frac{(n-1)p}{n(p-1)} }}{p\beta ^{\frac{p}{p-1}}} \left(\int_0^{\abs{\Omega}}f^\ast (s ) \, ds\right)^{\frac{p}{p-1}}.	
	\end{split} 
\end{equation*}

\noindent If we integrate by parts both sides of the last expression and set
$$\displaystyle{C=F(\abs{\Omega})\frac{\abs{\Omega}^{-\frac{p(n-1)}{n(p-1)} }}{p\beta ^{\frac{p}{p-1}}} \left(\int_0^{\abs{\Omega}}f^\ast (s ) \, ds\right)^{\frac{p}{p-1}}},$$
\noindent we obtain
\begin{equation}
\label{Solidus_Snake}
\tau \int_0^{\tau} \gamma_n t^{p-2} F (\mu(t)) \, dt + \tau H_\mu(\tau) \leq \int_{0}^{\tau} \int_0^t r^{p-2} F(\mu(r)) \, dr dt+ \int_0^{\tau} H_\mu(t) \, dt +C
\end{equation}
where
\[
H_\mu(\tau)=-\int_{\tau}^{+\infty} t^{p-2} \mu(t)^{-\frac{p(n-1)}{n(p-1)}} F(\mu(t)) \biggl( \int_0^{\mu(t)} f^*(s) \, ds \biggr)^{\frac{1}{p-1}} \, d\mu(t).
\]
Setting
\begin{equation*}
	\begin{multlined}
		\xi(\tau)=\int_0^\tau \int_0^t \gamma_n r^{p-2}F(\mu(r)) \, dr +\int_0^t H_{\mu}(t) \, dt  
	\end{multlined} 
\end{equation*}
then \eqref{Solidus_Snake} becomes
$$
\tau \xi'(\tau ) \leq \xi(\tau) +C.
$$
So lemma \ref{lemma_Gronwall}, with $\tau_0=v_m$, gives
$$\int_{0}^{\tau} \gamma_n t^{p-2} F(\mu(t)) \, dt + H_\mu(\tau) \le \left( \frac{\displaystyle{(p-1)\int_{0}^{v_m} t^{p-2} F(\mu(t) \, dt +H_\mu(v_m) +C}}{v_m}\right) \left( \frac{\tau}{v_m}\right)^{p-2}$$

\noindent Of course, the inequality holds as an equality if we replace $\mu(t)$ with $\phi(t)$, so we get, keeping in mind that  $\mu(t)\le \phi(t)= \abs{\Omega} $  for $t \leq v_m$,
$$
\int_{0}^{\tau} \gamma_n t^{p-2} F(\mu(t) \, dt + H_\mu(\tau)\le \int_{0}^{\tau} \gamma_n F(\phi(t)) \, dt +H_\phi(\tau)
$$
Letting $\tau \to \infty$, one has
$$
	\int_0^\infty t^{p-2} F(\mu(t)) dt \leq \int_0^\infty t^{p-2} F(\phi(t)) dt,
$$
as 
$H_\mu(\tau), H_\phi(\tau) \to 0$. This proves \eqref{scorciatoia}, and hence \eqref{fgen2}.

\noindent The fact that both $H_\mu$ and $H_\phi$ go to 0 as $\tau$ goes to infinity can be easily deduced distinguishing the cases. 
\begin{itemize}
	\item If $p \geq 2$
	\begin{align*}
	t^{p-2} \mu(t)&=\int_{u>t} t^{p-2} \, dx \leq \int_{u>t} u^{p-2} \, dx \leq \norma{u}_{L^p}^{p-2} \mu(t)^{\frac{2}{p}} \\
	\Rightarrow \abs{H_\mu(\tau)}&=\int_{\tau}^{+\infty} t^{p-2} F(\mu(t)) \mu(t)^{-\frac{p(n-1)}{n(p-1)}} \biggl( \int_0^{\mu(t)} f^*(s) \, ds \biggr) (-\mu'(t)) \, dt \\
	&\leq \biggl(\int_0^{\abs{\Omega}} f^*(s) \, ds\biggr) \norma{u}_{L^p}^{p-2} \int_{\tau}^{+\infty} F(\mu(t))  \mu(t)^{\frac{2}{p}-\frac{p(n-1)}{n(p-1)}-1} (-\mu'(t)) \, dt \xrightarrow{\tau \to +\infty} 0.
	\end{align*}
	
	\item If $p <2$
	\begin{align*}
	\abs{H_\mu(\tau)}&=\int_{\tau}^{+\infty} t^{p-2} F(\mu(t)) \mu(t)^{-\frac{p(n-1)}{n(p-1)}} \biggl( \int_0^{\mu(t)} f^*(s) \, s \biggr) (-\mu'(t)) \, dt \\
	&\leq \tau^{p-2} \int_{\tau}^{+\infty}  F(\mu(t)) \mu(t)^{-\frac{p(n-1)}{n(p-1)}} \biggl( \int_0^{\mu(t)} f^*(s) \, s \biggr) (-\mu'(t)) \, dt \xrightarrow{\tau \to +\infty} 0.
	\end{align*}
	and analogously for $H_\phi$, which concludes the proof. \qedhere
\end{itemize}
\end{proof}

\begin{proof}[Proof of Theorem \ref{teorema2}]{\color{white} .}

	\begin{enumerate}[$(i)$]
		\item Firstly, we observe that $\displaystyle{\int_0^{\mu(t)} f^\ast(s) \, ds= \mu(t)}$, so \eqref{3.2} becomes
		\begin{equation}
		\label{f=1}
		\gamma_n \mu(t)^{\left(1-\frac{1}{n}-\frac{1}{p}\right)\frac{p}{p-1}} \leq - \mu'(t) + \frac{1}{\beta ^{\frac{1}{p-1}}}\int_{\partial U_t^\text{ext}} \frac{1}{u} \, d\mathcal{H}^{n-1}(x).
		\end{equation}
		
		\noindent Let us multiply both sides by $t^{p-1} \mu(t)^\delta$, where $\delta = -\left( 1- \frac{1}{n}-\frac{1}{p}\right)\frac{p}{p-1}$. We point out that $\delta\geq 0$ for $p \leq \frac{n}{n-1}$. Hence, integrating from $0$ to $\tau \geq v_m$, we have
		\begin{equation}
		\label{20}
		\begin{split}
		\int_0^\tau \gamma_n t^{p-1} &\leq \int_0^\tau t^{p-1} \mu(t)^\delta (- \mu'(t)) \, dt + \frac{1}{\beta ^{ \frac{1}{p-1}}}\int_0^\tau t^{p-1}  \mu(t)^\delta\int_{\partial U_t^\text{ext}} \frac{1}{u} \, d\mathcal{H}^{n-1}(x) \\
		&\leq \int_0^\tau t^{p-1} \mu(t)^\delta (- \mu'(t)) \, dt +  \frac{\abs{\Omega}^{\delta+1}}{p \beta ^{\frac{p}{p-1}}}
		\end{split}
		\end{equation}
		
		\noindent Taking into account remark \ref{oss}, if we replace $\mu(t)$ with $\phi(t)$ the previous inequality holds as equality.
		
		\noindent Hence, we get
		\begin{equation*}
		\int_0^\tau t^{p-1} \mu(t)^\delta (- \mu'(t)) \, dt  \geq \int_0^\tau t^{p-1} \phi(t)^\delta (- \phi'(t)) \, dt .
		\end{equation*}
		
		\noindent Then an integration by parts gives
		\begin{equation*}
		-\tau^{p-1} \frac{\mu(\tau)^{\delta + 1}}{\delta +1} + (p-1) \int_0^\tau t^{p-2}\frac{\mu(t)^{\delta + 1}}{\delta +1} \, dt \geq 	-\tau^{p-1} \frac{\phi(\tau)^{\delta + 1}}{\delta +1} + (p-1) \int_0^\tau t^{p-2}\frac{\phi(t)^{\delta + 1}}{\delta +1}\, dt.
		\end{equation*}
		
		\noindent Finally, using Gronwall's Lemma with the function $\displaystyle{\xi(\tau)=\int_0^\tau s^{p-2}\left(\frac{\mu(s)^{\delta + 1}-\phi(s)^{\delta + 1}}{\delta +1}\right)\, ds}$ we obtain
		
		$$ \displaystyle{\tau^{p-2} \left(\frac{\mu^{\delta +1}(\tau)- \phi^{\delta+1}(\tau)}{\delta+1}\right)\le (p-1) \frac{\tau^{p-2}}{v_m^{p-2}}\int_0^{v_m} s^{p-2} \left(\frac{\mu^{\delta +1}(s)- \phi^{\delta+1}(s)}{\delta+1}\right) \, ds.  }$$
		
		 The quantity on the right-hand side is non-positive, thanks to \eqref{minima_eq}, so
		 \begin{equation*}
		\mu (\tau) \leq \phi (\tau) \quad \forall \tau \geq v_m.
		\end{equation*} 
		and, remembering that,
		$$
		\mu (\tau) \leq \phi (\tau) = \abs{\Omega} \quad \forall \tau \leq v_m,
		$$
		we get the point-wise inequality of the functions.
		
		\item Now we want to show that
		\[
		\norma{u}_{L^{k,1}(\Omega)} \leq \norma{v}_{L^{k,1}(\Omega^{\sharp})}
		\]
		so it is enough to show 
		\begin{equation}
		\label{k1}
		\int_0^{+\infty} \mu(t)^{\frac{1}{k}} \, dt \leq \int_0^{+\infty} \phi(t)^{\frac{1}{k}} \, dt
		\end{equation}
		
		\noindent We multiply \eqref{f=1} by $t^{p-1} \mu(t)^{\frac{1}{k}-\left( 1- \frac{1}{n}-\frac{1}{p}\right)\frac{p}{p-1}}$ and integrate from $0$ to $\tau \geq v_m$. Then using Lemma \ref{lemma3.3} and Remark \ref{oss}, we obtain
		\begin{equation}
		\label{Batman}
		\int_0^{\tau} \gamma_n t^{p-1} \mu(t)^{\frac{1}{k}} \, dt \leq \int_0^{\tau} t^{p-1} \mu(t)^{\frac{1}{k}-\left( 1- \frac{1}{n}-\frac{1}{p}\right)\frac{p}{p-1}} (-\mu'(t)) \, dt+ \frac{\abs{\Omega}^{\frac{1}{k}-\left( 1- \frac{1}{n}-\frac{1}{p}\right)\frac{p}{p-1} +1}}{p \beta^{\frac{p}{p-1}}}
		\end{equation}
		and equality holds if we replace $\mu$ with $\phi$.
		In order to be shorter, we set
		\[
		\eta=\frac{1}{k}-\left( 1- \frac{1}{n}-\frac{1}{p}\right)\frac{p}{p-1}, \quad C=\frac{\abs{\Omega}^{\eta +1}}{p \beta^{\frac{p}{p-1}}}.
		\]
		We point out that \eqref{Batman} follows by \eqref{20} if $\eta\ge  0$, namely
		$$0<k \leq \frac{n(p-1)}{n(p-1)-p}$$
		With these notations and keeping in mind that $\mu$ is a non increasing function, we have from \eqref{Batman} taht
		\begin{equation}
		\label{Big_Boss}
		\int_0^{\tau} \gamma_n t^{p-1} \mu(t)^{\frac{1}{k}} \, dt \leq \int_0^{\tau} -t^{p-1} \mu(t)^{\eta} \, d\mu(t)+C
		\end{equation}
		Let us set $\displaystyle{G(\ell)=\int_0^{\ell} w^{\eta} \, dw=\frac{\ell^{\eta+1}}{\eta+1}}$, let us integrate by parts both sides of \eqref{Big_Boss} in order to obtain
		\begin{equation}
		\label{Metal_Gear}
		\begin{split}
		    &\gamma_n \tau^{p-1} \int_0^{\tau} \mu(t)^{\frac{1}{k}} \, dt+\tau^{p-1} G(\mu(\tau))  \\
		    &\leq (p-1) \biggl[ \int_0^{\tau} \gamma_n t^{p-2} \int_0^t \mu(r)^{\frac{1}{k}} \, dr dt+\int_0^{\tau} t^{p-2} G(\mu(t)) \, dt \biggr] +C
		\end{split}
		\end{equation}
		Setting
		\[
		\xi(\tau)=\int_0^{\tau} \left(\gamma_n t^{p-2} \int_0^t \mu(r)^{\frac{1}{k}} \, dr\right) \, dt+\int_0^{\tau} t^{p-2} G(\mu(t)) \, dt
		\]
		\eqref{Metal_Gear} reads as follows
		\[
		\tau \xi'(\tau) \leq (p-1) \xi(\tau)+C
		\]
		Hence, using Gronwall's Lemma \ref{lemma_Gronwall} with $\tau_0=v_m$, we get
		\[
		\gamma_n \tau^{p-2} \int_0^{\tau} \mu(t)^{\frac{1}{k}} \, dt+\tau^{p-2} G (\mu(\tau)) \leq \biggl( \frac{(p-1)\xi(v_m)+C}{v_m} \biggr) \biggl( \frac{\tau}{v_m} \biggr)^{p-2}
		\]
		where 
		$$ \displaystyle{ \xi(v_m)=\int_0^{v_m} \gamma_n t^{p-2} \int_0^t \mu(r)^{\frac{1}{k}} \, dr \, dt+ \int_0^{v_m} t^{p-2} G(\mu(t)) \, dt     } $$
		Again, if we replace $\mu$ with $\phi$, the previous inequality holds as an equality and $\xi(v_m)$ is less or equal than the same quantity obtained by replacing $\mu$ with $\phi$, as \eqref{minima_eq} holds. Keeping in mind \eqref{mf}, we have
		\[
		\tau^{p-2}\left(\gamma_n  \int_0^{\tau} \mu(t)^{\frac{1}{k}} \, dt+ G (\mu(\tau)) \right) \leq \tau^{p-2} \left(\gamma_n  \int_0^{\tau} \phi(t)^{\frac{1}{k}} \, dt+G (\phi(\tau))\right)
		\]
		Passing to the limit as $\tau  \to + \infty$, we get
		\[
		\int_0^{+\infty} \mu(t)^{\frac{1}{k}} \, dt \leq  \int_0^{+\infty} \phi(t)^{\frac{1}{k}} \, dt
		\]
		namely \eqref{k1}.
		
		\vspace{1 em}
		
		\noindent To conclude the proof, we have to show that
		\[
		\norma{u}_{L^{pk,p}(\Omega)} \leq \norma{v}_{L^{pk,p}(\Omega^{\sharp})} \qquad \forall \, 0<k \leq \frac{n(p-1)}{n(p-1)-p}
		\]
		that is to say
		\[
		\int_0^{+\infty} t^{p-1} \mu(t)^{\frac{1}{k}} \, dt \leq \int_0^{+\infty} t^{p-1} \phi(t)^{\frac{1}{k}} \, dt.
		\]
		We consider \eqref{Big_Boss}, pass to the limit as $\tau \to +\infty$ and integrate by parts the first term on the right-hand side
		\[
		\int_0^{+\infty} \gamma_n t^{p-1} \mu(t)^{\frac{1}{k}} \, dt \leq (p-1) \int_0^{+\infty} t^{p-2} G(\mu(t)) \, dt +C.
		\]
		Hence it is enough to show that
		\[
		\int_0^{+\infty} t^{p-2} G(\mu(t)) \, dt \leq \int_0^{+\infty} t^{p-2} G(\phi(t)) \, dt.
		\]
		To this aim, we multiply \eqref{f=1} by $t^{p-1} G(\mu(t)) \mu(t)^{-\left(1-\frac{1}{n}-\frac{1}{p}\right)\frac{p}{p-1}} $ and integrate from $0$ to $\tau \geq v_m$
		\begin{align*}
		\int_0^{\tau} \gamma_n t^{p-1} G(\mu(t)) \, dt &\leq \int_0^{\tau} t^{p-1} G(\mu(t)) \mu(t)^{-\left(1-\frac{1}{n}-\frac{1}{p}\right)\frac{p}{p-1}}  \, d\mu(t) \\
		& + \frac{1}{\beta^{\frac{1}{p-1}}} \int_0^{\tau} t^{p-1} G(\mu(t)) \mu(t)^{-\left(1-\frac{1}{n}-\frac{1}{p}\right)\frac{p}{p-1}} \biggl( \int_{\partial U_t^{ ext }} \frac{1}{u} \, d\mathcal{H}^{n-1} \biggr) \, dt
		\end{align*}
		Since $\displaystyle{k \leq \frac{n(p-1)}{n(p-1)-p}}$, using Lemma \ref{lemma3.3} and the fact that the function $G(\ell)\ell^{-\left(1-\frac{1}{n}-\frac{1}{p}\right)\frac{p}{p-1}} $ is non decreasing, we obtain
		\begin{equation}
		\label{Ocelot}
		\int_0^{\tau} \gamma_n t^{p-1} G(\mu(t)) \, dt \leq \int_0^{\tau} t^{p-1} G(\mu(t)) \mu(t)^{-\left(1-\frac{1}{n}-\frac{1}{p}\right)\frac{p}{p-1}}  \, d\mu(t)+ C
		\end{equation}
		with
		\[
		C=\frac{1}{p \beta^{\frac{p}{p-1}}} G (\abs{\Omega}) \abs{\Omega}^{-\left(1-\frac{1}{n}-\frac{1}{p}\right)\frac{p}{p-1}+1} 
		\]
		If we replace $\mu$ with $\phi$ the previous inequality holds as an equality, thanks to \eqref{oss}. 
		Now, let us integrate by parts both sides of \eqref{Ocelot},  obtaining
		\begin{equation}
		\label{Solid_Snake}
		\tau \int_0^{\tau} \gamma_n t^{p-2} G(\mu(t)) \, dt+\tau H(\tau) \leq \int_0^{\tau} \int_0^t \gamma_n t^{p-2} G(\mu(r)) \, dr dt+\int_0^{\tau} H_\mu(t) \, dt+C
		\end{equation}
		where 
		\[
		H_\mu(\tau)=-\int_{\tau}^{+\infty} t^{p-2} G(\mu(t)) \mu(t)^{-\left(1-\frac{1}{n}-\frac{1}{p}\right)\frac{p}{p-1}}  \, d\mu(t)
		\]
		Setting
		\[
		\xi(\tau)=\int_0^{\tau} \int_0^t \gamma_n t^{p-2} G(\mu(r)) \, dr dt+\int_0^{\tau} H_\mu(t) \, dt
		\]
		the \eqref{Solid_Snake} reads as follows
		\[
		\tau \xi'(\tau) \leq \xi(\tau)+C
		\]
		Again, using Gronwall's Lemma \ref{lemma_Gronwall}, we get
		$$ \displaystyle{\int_0^\tau \gamma_n t^{p-2} G(\mu(t)) \,  dt+ H_\mu(\tau) \le \left( \frac{(p-1) \xi(v_m) +C}{v_m} \right)  \left( \frac{\tau}{v_m}\right)^{p-2} }$$
		with
		$$ \xi(v_m)= \int_0^{v_m} \int_0^t \gamma_n t^{p-2} G(\mu(r)) \, dr dt+\int_0^{v_m} H_\mu(t) \, dt.
 $$
		
		Keeping in mind that for $\phi$ the previous inequalities hold as equality and the fact that $G$ is not decreasing, $\xi(v_m)$ is less or equal to the same quantity obtained by replacing $\mu$ with $\phi$. Hence, we obtain
		$$\displaystyle{ \int_0^\tau \gamma_n t^{p-2} G(\mu(t)) \,  dt+ H_\mu(\tau) \le \int_0^\tau \gamma_n t^{p-2} G(\phi(t)) \,  dt+ H_\phi(\tau)  }$$
		and passing to the limit as $\tau \to +\infty$, we finally get
		\[
		\int_0^{+\infty} t^{p-2} G(\mu(t)) \, dt \leq  \int_0^{+\infty} t^{p-2} G(\phi(t)) \, dt
		\]
		indeed, as in the proof of Theorem \ref{teorema1}, $H_\mu(\tau)$ and $H_\phi(\tau)$ go to 0 as $\tau \to \infty$.
		That concludes the proof. \qedhere
	\end{enumerate}
	
\end{proof}

%

\begin{corollario}
\label{corollario2}
	Let $u$ and $v$ be the solutions to \eqref{p_originale} and \eqref{p_tondo} respectively. Then, if $p \geq n$, we have
	\begin{equation*}
	\norma{u}_{L^1(\Omega)}\le \norma{v}_{L^1(\Omega^\sharp)} \quad \text{and} \quad \norma{u}_{L^p(\Omega)}\le \norma{v}_{L^p(\Omega^\sharp)}.
	\end{equation*}
	
	Moreover in the case $f \equiv 1$, Theorem \ref{teorema2} gives 
\begin{equation*}
    \norma{u}_{L^1(\Omega)}\le \norma{v}_{L^1(\Omega^\sharp)} \quad \text{and} \quad \norma{u}_{L^p(\Omega)}\le \norma{v}_{L^p(\Omega^\sharp)} \qquad \forall p>1
\end{equation*}

and the point-wise comparison for $\displaystyle{p \leq \frac{n}{n-1}}$.
\end{corollario}
\begin{proof}
    If $p \geq n$ the upper bounds of $k$, in both cases \eqref{diseq_f_generica} e \eqref{fgen2}, are greater than $1$ and so we can choose $k=1$. The assertion follows from the fact that
    $$
    \norma{\cdot}_{L^{p,p}(\Omega)} = \norma{\cdot}_{L^{p}(\Omega)}.
    $$
    
    \noindent Analogously if $f \equiv 1$.
\end{proof}

\section{Faber–Krahn inequality}
\label{section_3}

We recall that the first eigenvalue of $p$-Laplace operator with Robin boundary conditions is obtained as the minimum of the Rayleigh quotients, i.e.,

\begin{equation}
\label{trace}
\lambda_{1,\beta} (\Omega)= \min_{\substack{\omega \in W^{1,p}(\Omega) \\ \omega\ne 0}} \frac{\displaystyle{\int_{\Omega} \abs{\nabla \omega}^p \, dx + \beta \int_{\partial \Omega} \abs{\omega}^p \, d\mathcal{H}^{n-1}(x)}}{\displaystyle{\int_{\Omega} \abs{\omega}^p \, dx}}.
\end{equation}

We can observe that if $u$ achieves the minimum of Rayleigh quotients, so does $\abs{u}$. From this we have that $u$ in non-negative. Furthermore, we have if $u_{\lambda_{1}} \geq 0$, as a consequence of Harnack inequality, $u_{\lambda_{1}} >0$.

Another important thing is that the eigenvalue is simple. Indeed, as shown in \cite{L}, if $\Omega$ is smooth enough and $u$ and $v$ are to eigenfunctions referred to the first eigenvalue, we can choose as test function $\displaystyle{\varphi_1 = \frac{u^p - v^p}{ u^{p-1}}}$ in
$$
\int_{\Omega} \abs{\nabla u}^{p-2}\nabla u \nabla \varphi_1 \, dx + \beta \int_{\partial \Omega} u^{p-1} \varphi_1 \, d\mathcal{H}^{n-1}(x) = \int_{\Omega} \lambda_1 u^{p-1} \varphi_1 \, dx
$$
and $\displaystyle{\varphi_2 = \frac{v^p - u^p}{ v^{p-1}}}$ in 
$$
\int_{\Omega} \abs{\nabla v}^{p-2}\nabla v \nabla \varphi_2 \, dx + \beta \int_{\partial \Omega} v^{p-1} \varphi_2 \, d\mathcal{H}^{n-1}(x) = \int_{\Omega} \lambda_1 v^{p-1} \varphi_2 \, dx.
$$

\noindent Summing the two equations, we have 
\begin{equation*}
    \begin{split}
       0&= \int_\Omega \left\lbrace 1 + (p-1) \left(\frac{v}{u}\right)^p\right\rbrace \abs{\nabla u}^p + \left\lbrace 1 + (p-1) \left(\frac{u}{v}\right)^p\right\rbrace \abs{\nabla v}^p\\
       & - \int_{\Omega} p \left( \frac{v}{u} \right)^{p-1} \abs{\nabla u}^{p-2} \nabla u \nabla v +p \left( \frac{u}{v} \right)^{p-1} \abs{\nabla v}^{p-2} \nabla v \nabla u \\
       & = \int_{\Omega} (u^p-v^p) \left( \abs{\nabla \log{u}}^p -\abs{\nabla \log{v}}^p\right) \\
        &- \int_\Omega p v^p \abs{\nabla \log{u}}^{p-2} \abs{\nabla \log{u}} \left( \nabla\log{v} - \nabla\log{u} \right) \\
       & - \int_\Omega p u^p \abs{\nabla \log{v}}^{p-2} \abs{\nabla \log{v}} \left( \nabla\log{u} - \nabla\log{v} \right)
    \end{split}
\end{equation*}
Now, using the well known inequalities, which hold true for each $w_1$ and $w_2\in \R^n$ ,
\begin{equation}
\label{lindt}
    \begin{gathered}
        \abs{w_2}^p\ge \abs{w_1}^p+ p \abs{w_1}^{p-2}w_1\cdot (w_2-w_1)+ \frac{\abs{w_2-w_1}^p}{2^{p-1}-1} \quad \text{if} \quad p\ge 2\\
        \abs{w_2}^p\ge \abs{w_1}^p+ p \abs{w_1}^{p-2}w_1\cdot (w_2-w_1)+ C(p)\frac{\abs{w_2-w_1}^2}{(\abs{w_1}+\abs{w_2})^{2-p}} \quad \text{if} \quad 1<p< 2,
    \end{gathered}
\end{equation}
if we consider the case $p \geq 2$, we choose $w_2= \nabla\log u$ and $w_1= \nabla\log v$, we obtain 
\begin{equation*}
    \frac{1}{2^{p-1}-1} \int_\Omega \left( \frac{1}{v^p}+ \frac{1}{u^p}\right) \abs{v \nabla u - u \nabla v}^p =0.
\end{equation*}
Hence, we obtain that $v \nabla u = u \nabla v$ a.e. in $\Omega$, and so there exists a constant $K$ for which $u=Kv$. This means that $\lambda_1$ is simple. 

\noindent For the proof of \eqref{lindt}, we refer to \cite{L}.

\noindent The following corollary of Theorem \ref{teorema1} holds true, that is Faber-Krahn inequality.

\begin{corollario}
	Let $u$ and $v$ be the solutions to \eqref{p_originale} and \eqref{p_tondo}, respectively. Then , if $p \geq n$, we have
	
	\begin{equation*}
	\lambda_{1,\beta} (\Omega) \geq \lambda_{1,\beta}(\Omega^\sharp).
	\end{equation*}
\end{corollario}
\begin{proof}
	Let $u$ an eigenfunction referred to the first eigenvalue of \eqref{p_originale}, then it solves
	\begin{equation*}
	\begin{cases}
	-\Delta_p u= \lambda_{1,\beta}(\Omega) \, \abs{u}^{p-2} u & \text{ in } \Omega \\
	\abs{\nabla u}^{p-2} \displaystyle{\frac{\partial u}{\partial \nu}} + \beta  \abs{u}^{p-2}u =0  & \text{ on } \partial \Omega.
	\end{cases}
	\end{equation*}
	Now, let $z$ be a solution to the following problem 
	\begin{equation*}
	\begin{cases}
	-\Delta_p z= \lambda_{1,\beta}(\Omega) \, \lvert u^\sharp \rvert^{p-2} u^\sharp & \text{ in } \Omega^\sharp \\
	\abs{\nabla z}^{p-2} \displaystyle{\frac{\partial z}{\partial \nu}} + \beta  \abs{z}^{p-2}z =0  & \text{ on } \partial \Omega^\sharp.
	\end{cases}
	\end{equation*} 
	\noindent In that case, corollary \ref{corollario2} gives
	
	$$
	\int_{\Omega} \abs{u}^p \, dx = \int_{\Omega^\sharp}\abs{u^\sharp}^p \, dx \leq \int_{\Omega^\sharp} \abs{z}^p \, dx,
	$$
	and hence, by H\"older inequality 
	$$
	\int_{\Omega^\sharp}(u^\sharp)^{p-2} u^\sharp z \, dx\leq \left(\int_{\Omega^\sharp}\abs{u^\sharp}^p\, dx \right)^{\frac{p-1}{p}}  \left(\int_{\Omega^\sharp} z^p \, dx\right)^{\frac{1}{p}}\leq \int_{\Omega^\sharp} z^p \, dx. 
	$$
	
	\noindent Therefore, observing that we can write the eigenvalue $\lambda_{1,\beta}(\Omega)$ in the following way, we obtain
	\begin{equation*}
	\begin{split}
	\lambda_{1,\beta}(\Omega)&= \frac{\displaystyle{\int_{\Omega^\sharp} \abs{\nabla z}^p \, dx + \beta \int_{\partial \Omega^\sharp} z^p \, d\mathcal{H}^{n-1}(x)}}{\displaystyle{\int_{\Omega^\sharp} (u^\sharp)^{p-2} u^\sharp  z \, dx}} \\
	&\geq \frac{\displaystyle{\int_{\Omega^\sharp} \abs{\nabla z}^p \, dx + \beta \int_{\partial \Omega^\sharp} z^p \, d\mathcal{H}^{n-1}(x)}}{\displaystyle{\int_{\Omega^\sharp} z^p \, dx}} \geq\lambda_{1,\beta}(\Omega^\sharp).
	\end{split}
	\end{equation*}
\end{proof}

 \section{Conclusions}
\label{section_4}

We have been able to extend the results obtained for the Laplacian to the $p$-Laplacian.
Many problems remain open, such as

\textbf{Open Problem}  In the assumptions of Theorem \ref{teorema2}, does the point-wise comparison hold also for $p>\frac{n}{n-1}$?

We have already observed in the corollary \ref{corollario2} that if $p\geq n$ we have an estimate on the $L^p$ norms of $u$ and $v$. Can we generalize this estimate also for $q\neq p$? We know for sure that for $q=\infty$ this can't be done, as it can be seen in the following example. 
\begin{esempio} Let $\Omega\subseteq \R^n$ be the union of two disjoint balls,
$B_1$ and $B_r$ with radii 1 and $r$ respectively. We choose $\displaystyle{\beta< \left( \frac{n-1}{p-1} \right)^{p-1} }$ with $p\neq n $, and we fix $f=1$ on $B_1$ and $f=0$ on $B_r$. Both $u$ and $v$ can be explicitly computed. We have $\norma{u}_\infty - \norma{v}_\infty = C r^n + o(r^n)$, where $C$ is a positive constant. 

\end{esempio} 
\noindent \emph{Proof.} 
 We want an explicit expression of $u$ and $v$ respectively. 
\noindent Starting from $u$, it is a solution to
\begin{equation*}
\begin{cases}
-\div(\abs{\nabla u}^{p-2} \nabla u)= f & \text{ in } \Omega \\
\abs{\nabla u}^{p-2} \displaystyle{\frac{\partial u}{\partial \nu}} + \beta  \abs{u}^{p-2}u =0  & \text{ on } \partial \Omega.
\end{cases}
\end{equation*}
with $f\rvert_{B_1}= 1$ and $f\rvert_{B_r}= 0$.

\noindent It's clear that $u \rvert_{B_r}=0$ and $u(x)=u(\abs{x})$ it's radial on $B_1$.

\noindent So the equation \eqref{p_originale} becomes
$$
s^{n-1} \Delta_p u(s)= \frac{d}{ds}\left( s^{n-1} \abs{u'(s)}^{p-2} u'(s)\right)
$$
and then
\begin{align*}
    \frac{d}{ds} \left( s^{n-1} \abs{u'(s)}^{p-2} u'(s)\right) &=s^{n-1} \Delta_p u(s)=-s^{n-1} \\
    s^{n-1} \abs{u'(s)}^{p-2} u'(s) &=- \frac{s^n}{n} + c. \\
\end{align*}

\noindent We set $c=0$, in order to have a $C^1$-solution. 
$$
\abs{u'(s)}^{p-2} u'(s)=- \frac{s}{n} \implies  u'(s) = - \frac{s^{\frac{1}{p-1}}}{n^{\frac{1}{p-1}}} \qquad \alpha=\frac{1}{p-1}.
$$
If we integrate, we obtain
$$
u(s)= - \frac{p-1}{n^\alpha  p } s^{\frac{p}{p-1}} + A.
$$
The Robin boundary conditions become
$$
\abs{u'(1)}^{p-2} u'(1) + \beta u(1)^{p-1}=0 \quad (u\ge0),
$$
now we can compute the value of $A$
$$
-\frac{1}{n} + \beta \left( - \frac{p-1}{n^\alpha  p } + A \right)^{p-1}=0 \implies A= \frac{1}{(n\beta)^\alpha} + \frac{p-1}{n^\alpha  p }.
$$
So
	\begin{equation*}
u(s)=  \frac{p-1}{n^\alpha  p }\left(1-s^{\frac{p}{p-1}}\right)  + \frac{1}{(n \beta)^\alpha}.
	\end{equation*}
As $u$ is decreasing, we have 
$$
\norma{u}_\infty =  u(0)= \frac{p-1}{n^\alpha  p }+\frac{1}{(n \beta)^\alpha}. 
$$

\noindent Now, let us compute $v(s)$. We will do this firstly for $s\in (0,1)$, then for $s\in (1, \overline{r})$ where $\overline{r}= (1 + r^n)^{\frac{1}{n}}$ is determined by the condition $\abs{\Omega}= \vert\Omega^\sharp\vert$.

\noindent Let $s<1$
$$
\frac{d}{ds}\left( s^{n-1} \abs{v'(s)}^{p-2} v'(s)\right)=-s^{n-1}
$$
$$
\abs{v'(s)}^{p-2} v'(s)=- \frac{s}{n} \implies v'(s) = - \frac{s^{\frac{1}{p-1}}}{n^\alpha}
$$
$$
v(s)= - \frac{p-1}{n^\alpha  p } s^{\frac{p}{p-1}} + B.
$$
Now we can't determine $B$ as before, as $v$ is not identically 0 in the anulus $B_{\overline{r}} \backslash B_1$.

\noindent Let $s >1$ and $p \neq n $
$$
\frac{d}{ds}\left( s^{n-1} \abs{v'(s)}^{p-2} v'(s)\right)=0
$$
$$
\abs{v'(s)}^{p-2} v'(s)=\frac{C}{s^{n-1}} 
$$
by imposing the continuity of the derivative for $s=1$, we obtain that $C= -1/n$
\begin{align*}
    v'(s) &= - \frac{s^{-\frac{n-1}{p-1}}}{n^\alpha}, \\
    v(s) &= - \frac{p-1}{n^\alpha  (p-n)} s^{\frac{p-n}{p-1}} + D,
\end{align*}
and by Robin conditions
\begin{gather*}
    \abs{v'(\overline{r})}^{p-2} v'(\overline{r}) + \beta v(\overline{r})^{p-1}=0, \\
    - \frac{\overline{r}^{-n-1}}{n} + \beta \left( - \frac{p-1}{n^\alpha  (p-n)} \overline{r}^{\frac{p-n}{p-1}} + D \right)^{(p-1)}=0, \\
    D= \frac{1}{(n\beta)^\alpha}\overline{r}^{-\frac{n-1}{p-1}}+ \frac{p-1}{n^\alpha  (p-n)} \overline{r}^{\frac{p-n}{p-1}}.
\end{gather*}

\noindent By imposing the continuity of $v$ for $s=1$, we have 
$$
B= \frac{p-1}{n^\alpha p}+ \frac{\overline{r}^{-\frac{n-1}{p-1}}}{(n\beta)^\alpha}+ \frac{p-1}{n^\alpha  (p-n)}\left( \overline{r}^{\frac{p-n}{p-1}} -1\right)
$$
that is to say
\begin{equation*}
    v(s)= \begin{cases}
	u(s) +\frac{1}{(n\beta)^\alpha} \left( \overline{r}^{-\frac{n-1}{p-1}}-1\right)+ \frac{p-1}{n^\alpha  (p-n)}\left( \overline{r}^{\frac{p-n}{p-1}} -1\right) & \text{ if } s<1 \\
    \frac{1}{(n\beta)^\alpha} \overline{r}^{-\frac{n-1}{p-1}}+ \frac{p-1}{n^\alpha  (p-n)} \left( \overline{r}^{\frac{p-n}{p-1}}-s^{\frac{p-n}{p-1}}\right)& \text{ if } 1<s<\overline{r}
	\end{cases}
\end{equation*}
For convenience's sake, we set $\displaystyle{h=\frac{1}{(n\beta)^\alpha} \left( \overline{r}^{-\frac{n-1}{p-1}}-1\right)+ \frac{p-1}{n^\alpha  (p-n)}\left( \overline{r}^{\frac{p-n}{p-1}} -1\right)}$.
	
\noindent So we have
$$
\norma{v}_{L^\infty(\Omega^\sharp) } = \norma{v}_{L^\infty(B_1) } =  \norma{u}_{L^\infty(\Omega) } + h= u(0)+h
$$
	
\noindent By using Taylor expansion of the function $(1+r^n)^{\delta}$ we get
$$
h= \left(- \frac{1}{(n\beta)^\alpha} \frac{n-1}{n(p-1)} + \frac{1}{n^{\alpha +1}}\right)r^n + o(r^n),
$$
so, if we choose $\beta < \left( \frac{n-1}{p-1}\right)^{p-1} $, we get
$$
\norma{v}_{L^\infty(\Omega^{\sharp})}= \norma{u}_{L^\infty (\Omega)} - C r^n + o(r^n) \text{ where } C>0.
$$
\vspace{1 em}

\noindent Next example \ref{esempio2} is a counterexample to the corollary \ref{corollario2} in the case $n > p$.

\begin{esempio} 
\label{esempio2}
Let $\Omega \subseteq \R^n$, $p<n$ be the union of two disjoint balls $B_1$ and $B_r$ with radii 1 and $r$ respectively. We choose $\displaystyle{\beta\leq \left( \frac{n-p}{p(p-1)} \right)^{p-1} }$ and we fix $f=1$ on $B_1$ and $f=0$ on $B_r$. Both $u$ and $v$ can be explicitly computed. We have $\norma{u}_p^p - \norma{v}_p^p = C r^n + o(r^n)$, where $C$ is a positive constant. 

\end{esempio}
\noindent \emph{Proof.}
Let us consider the Taylor expansion of $(1+y)^p$, we get
$$
\norma{v}^p_{L^p(B_1)} = \int_{B_1} (u+h)^p=  \norma{u}^p_{L^p(B_1)} +p\norma{u}^{p-1}_{L^{p-1}(B_1)} h + o(r^n)
$$
Moreover
$$
\norma{v}^p_{L^p(B_{\overline{r}} \backslash B_1)}= \frac{\omega_n }{(n\beta)^{\alpha p}} r^n + o(r^n)
$$
as if $1<s<\overline{r}$
\[
    \frac{1}{(n\beta)^\alpha} \overline{r}^{-\frac{n-1}{p-1}} \leq  v(s) \leq \frac{1}{(n\beta)^\alpha} \overline{r}^{-\frac{n-1}{p-1}} + \frac{p-1}{n^\alpha  (p-n)} \left( \overline{r}^{\frac{p-n}{p-1}}-1\right)
\]
thus
$$
v(s)= \frac{1}{(n\beta)^\alpha} + O(r^n)
$$
and by integration we obtain the value of the norm in $L^p(B_{\overline{r}} \backslash B_1)$.

\noindent So
$$
\norma{v}^p_{L^p(\Omega^{\sharp})} =\norma{v}^p_{L^p(B_1)} + \norma{v}^p_{L^p(B_{\overline{r}} \backslash B_1)} =   \norma{u}^p_{L^p(\Omega)} +p\norma{u}^{p-1}_{L^{p-1}(B_1)} h + \frac{\omega_n }{(n\beta)^{\alpha p}} r^n +o(r^n) 
$$
and recalling that
$$
	h= \left(- \frac{1}{(n\beta)^\alpha} \frac{n-1}{n(p-1)} + \frac{1}{n^{\alpha +1}}\right)r^n + o(r^n)
$$
we get
$$
\norma{v}^p_{L^p(\Omega^{\sharp})} =\norma{u}^p_{L^p(\Omega)} +\left[ p\norma{u}^{p-1}_{L^{p-1}(B_1)} \left(- \frac{1}{(n\beta)^\alpha} \frac{n-1}{n(p-1)} + \frac{1}{n^{\alpha +1}} \right)+ \frac{\omega_n }{(n\beta)^{\alpha p}} \right] r^n +o(r^n).
$$
We have to understand whether 
\begin{equation}
\label{38}
    p\norma{u}^{p-1}_{L^{p-1}(B_1)} \left(- \frac{1}{(n\beta)^\alpha} \frac{n-1}{n(p-1)} + \frac{1}{n^{\alpha +1}} \right)+ \frac{\omega_n }{(n\beta)^{\alpha p}} <0.
\end{equation}

\noindent If we choose $\displaystyle{\beta< \left( \frac{n-1}{p-1} \right)^{p-1} }$ we have  $\displaystyle{- \frac{1}{(n\beta)^\alpha} \frac{n-1}{n(p-1)} + \frac{1}{n^{\alpha +1}}<0}$. In order to have \eqref{38}, we need
 
\begin{align*}
    \norma{u}^{p-1}_{L^{p-1}(B_1)} &> \frac{\omega_n }{(n\beta)^{\alpha p}} \left[ \frac{n(p-1) n^\alpha \beta^\alpha }{ p(n-1)-p\beta^\alpha (p-1)} \right] \\
    \norma{u}^{p-1}_{L^{p-1}(B_1)} &> \frac{\omega_n }{(n\beta)^{\alpha (p-1)}} \left[ \frac{n(p-1) }{ p(n-1)-p\beta^\alpha (p-1)} \right].
\end{align*}

If we show that
\begin{equation}
    \label{39}
    \left[ \frac{n(p-1) }{ p(n-1)-p\beta^\alpha (p-1)} \right] \leq 1
\end{equation}

\noindent then
$$
 u(s) > \frac{1}{(n \beta)^\alpha } \implies \norma{u}^{p-1}_{L^{p-1}(B_1)} > \frac{\omega_n}{(n\beta)^{\alpha (p-1)}}.
$$
We just have to verify \eqref{39} 
\begin{equation*}
    \begin{split}
\left[ \frac{n(p-1) }{ p(n-1)-p\beta^\alpha (p-1)} \right] \leq 1 &\iff n(p-1) \leq p(n-1) - p \beta^\alpha (p-1)\\
&\iff p-n \leq - p(p-1) \beta^\alpha <0 \quad \text{ (if and only if $p<n!$)} \\
&\iff \beta \leq \left( \frac{n-p}{p(p-1)}\right)^{p-1}
    \end{split}
\end{equation*}

\renewcommand{\abstractname}{}    
	\vspace{0.5cm}
\begin{abstract}

		\noindent 	\textsc{Dipartimento di Matematica e Applicazioni “R. Caccioppoli”, Universita` degli Studi di Napoli “Federico II”, Complesso Universitario Monte S. Angelo, via Cintia - 80126 Napoli, Italy.}
		
		\textsf{e-mail: vincenzo.amato@unina.it}
	
		\textsf{e-mail: albalia.masiello@unina.it}
		
	\vspace{0.5cm}
	\noindent \textsc{Mathematical and Physical Sciences for Advanced Materials and Technologies, Scuola Superiore Meridionale, Largo San Marcellino 10, 80126 Napoli, Italy.}
	
		\textsf{e-mail: andrea.gentile2@unina.it}
\end{abstract}


\addcontentsline{toc}{chapter}{ Bibliografia}
\begin{thebibliography}{9}
\bibitem{ACNT}
   Alvino A., Chiacchio F., Nitsch C., and Trombetti C. Sharp estimates for solutions to elliptic problems with mixed boundary conditions. \emph{J. Math. Pures Appl.}, 152:251—261 (2021).

\bibitem{ALT}
 A. Alvino, P.L. Lions, and G. Trombetti. Comparison results for elliptic and parabolic equations via Schwarz symmetrization. \emph{Ann. Inst. H. Poincaré Anal. Non Linéaire}, 7(2):37–65 (1990).

\bibitem{ANT}
A. Alvino, C. Nitsch, and C. Trombetti. A Talenti comparison result for solutions to elliptic problems with robin boundary conditions. \emph{to appear on Comm. Pure Appl. Math.}

\bibitem{BD}
Dorin Bucur and Daniel Daners. An alternative approach to the Faber-Krahn inequality for Robin problems. \emph{Calc. Var. Partial Differential Equations}, 37(1-2):75–86 (2010).

\bibitem{BG}
Dorin Bucur and Alessandro Giacomini. A variational approach to the isoperimetric inequality for the Robin eigenvalue problem. \emph{Arch. Ration. Mech. Anal.}, 198(3):927–961 (2010).

\bibitem{BG2}
Dorin Bucur and Alessandro Giacomini. Faber-Krahn inequalities for the Robin-Laplacian:
a free discontinuity approach. \emph{Arch. Ration. Mech. Anal.}, 218(2):757–824 (2015).

\bibitem{BGT} 
Dorin Bucur, Alessandro Giacomini, and Paola Trebeschi. Best constant in Poincaré inequalities with traces: a free discontinuity approach. \emph{Ann. Inst. H. Poincaré Anal. Non Linéaire}, 36(7):1959–1986 (2019).

\bibitem{L}
Peter Lindqvist. Addendum: “On the equation ${\rm div}(|\nabla u|^{p-2}\nabla u)+\lambda|u|^{p-2}u=0$” [Proc. Amer. Math. Soc. 109 (1990), no. 1, 157–164; MR1007505 (90h:35088)]. \emph{Proc. Amer. Math. Soc.}, 116(2):583–584 (1992).

\bibitem{N} 
Jindřich Nečas. Direct methods in the theory of elliptic equations. Springer Monographs in Mathematics. \emph{Springer, Heidelberg} (2012). Translated from the 1967 French original by Gerard Tronel and Alois Kufner, Editorial coordination and preface by Šárka Nečasová and a contribution by Christian G. Simader.

\bibitem{T}
Giorgio Talenti. Elliptic equations and rearrangements. \emph{Ann. Scuola Norm. Sup. Pisa Cl. Sci. (4)}, 3(4):697–718 (1976).

\bibitem{T2} 
Giorgio Talenti. Nonlinear elliptic equations, rearrangements of functions and Orlicz spaces. \emph{Ann. Mat. Pura Appl. (4)}, 120:160–184 (1979).

\bibitem{T3}
Giorgio Talenti. Inequalities in rearrangement invariant function spaces. In \emph{Nonlinear anal- ysis, function spaces and applications, Vol. 5 (Prague, 1994)}, pages 177–230. Prometheus, Prague (1994).

\bibitem{V}
J. L. Vázquez. A strong maximum principle for some quasilinear elliptic equations. \emph{Appl. Math. Optim.}, 12(3):191–202 (1984).

\end{thebibliography}
\end{document}